\documentclass[11pt]{amsart}
\usepackage{amsthm,amsmath,amstext,amssymb,extsizes,xcolor,scrextend}
\usepackage[all]{xy}
\usepackage[active]{srcltx}
\usepackage{tikz}
\usetikzlibrary{decorations.shapes}
\usetikzlibrary{decorations.text}
\usetikzlibrary{calc}
\usepackage{xcolor}

\usepackage[utf8]{inputenc}

\usepackage[lmargin=80pt,rmargin=80pt,tmargin=80pt,bmargin=80pt]{geometry}

\usepackage{comment}
\usepackage{caption}
\usepackage{subcaption}

\usepackage{hyperref}
\usepackage[normalem]{ulem}

\usepackage{transparent}

\usepackage{array}
\usepackage{palatino}

\usepackage{pst-all}
\usepackage{pstricks}
\usepackage{enumerate}
\usepackage{lipsum}
\usepackage{float}
\usepackage{multicol}
\usepackage{mathrsfs}
\usepackage{longtable}

\theoremstyle{definition}
\newtheorem{definition}{Definition}

\theoremstyle{plain}
 \newtheorem{thm}{Theorem}
 \newtheorem{cor}{Corollary}
 
 \newtheorem{prop}{Proposition}
\theoremstyle{remark}
\newtheorem{remark}{Remark}

\usepackage{multicol}\usepackage{array}
\newcommand{\bea}{\begin{eqnarray}}
\newcommand{\eea}{\end{eqnarray}}

\newtheorem{lemma}{Lemma}

\newcommand{\beq}{\begin{equation}}
\newcommand{\eeq}{\end{equation}}
\newcommand{\enn}{\nonumber \end{equation}}

 \newcommand{\rk}{{\rm r}\,}

 \newcommand{\cF}{\mathcal{F}}

 \newcommand{\cB}{\mathcal{B}}

\title[On a conjecture of Gross, Mansour and Tucker for $\Delta$-matroids]{On a conjecture of Gross, Mansour and Tucker for $\Delta$-matroids}

\usepackage{hyperref}
\hypersetup{
	colorlinks=true,	
	linkcolor=blue,	
	citecolor=[rgb]{0.0, 0.5, 0.0},	
	filecolor=magenta,	
	urlcolor=blue	
}

\author[R\'emi Cocou Avohou]{R\'emi Cocou Avohou}

\address[1]{Okinawa Institute of Science and Technology Graduate University, 1919-1, Tancha, Onna, Kunigami District, Okinawa 904-0495, Japan, \& ICMPA-UNESCO Chair, 072BP50, Cotonou, \& Ecole Normale Superieure, B.P 72, Natitingou, Benin,\newline {\tt \href{mailto:remi.avohou@oist.jp}{remi.avohou@oist.jp}}}

\begin{document}

\maketitle

\begin{abstract}
Gross, Mansour, and Tucker introduced the partial-duality polynomial of a ribbon graph [Distributions, European J. Combin. 86, 1--20, 2020], the generating function enumerating partial duals by Euler genus. Chmutov and Vignes-Tourneret wondered if this polynomial and its conjectured properties would hold for general delta-matroids, which are combinatorial abstractions of ribbon graphs. Yan and Jin contributed to this inquiry by identifying a subset of delta-matroids--specifically, even normal binary ones--whose twist polynomials are characterized by a singular term. Building upon this foundation, the current paper expands the scope of investigation to encompass even non-binary delta-matroids, revealing that none of them have width-changing twists.
\end{abstract}

\section{Introduction}
Chmutov introduced partial duality, a generalization of geometric duality for ribbon graphs, inspired by the Bollob\'as-Riordan and Jones-Kauffman polynomials \cite{MR2507944}. Partial duality allows one to dualize only some edges of a ribbon graph, and obtain a partial dual. Gross, Mansour, and Tucker defined the partial-dual genus polynomial as a generating function that enumerates the partial duals of a ribbon graph by their genus \cite{MR4056111}. This polynomial has a counterpart for delta-matroids, which are combinatorial models of ribbon graphs.

Delta-matroids, introduced by Bouchet \cite{MR2507944}, are a generalization of matroids that capture the essence of graph theory. A delta-matroid has feasible sets, analogous to bases of a matroid, which can have different sizes but satisfy the Symmetric Exchange Axiom. Delta-matroids can also encode information about how a graph is embedded on a surface. Bouchet \cite{MR1020647} showed that ribbon graphs, which are graphs with cyclic ordering of edges around each vertex, have delta-matroids that reflect their properties. For example, quasi-trees, which are subgraphs with one boundary cycle, are the same as spanning-trees, which are genus-zero spanning ribbon subgraphs. The edge set and the spanning quasi-trees of a ribbon graph form a delta-matroid, as a surprising result. A key connection between the two theories is the twist operation on delta-matroids, which corresponds to partial duality on ribbon graphs. This operation has many implications for the delta-matroid polynomial, which is a generalization of the Tutte polynomial.

Some graph polynomials, such as the Tutte polynomial, are better viewed as matroid polynomials, since they depend only on the matroid structure of the graph. Recently, there has been a lot of interest in extending the Tutte polynomial to graphs embedded on surfaces. Three such extensions are the Las Vergnas polynomial, the Bollob\'as-Riordan polynomial, and the Kruskal polynomial, which are defined for embedded graphs. These polynomials have been further generalized to delta-matroids as delta-matroid polynomials 
Gross, Mansour, and Tucker \cite{MR4056111} defined the partial-dual Euler genus polynomials and the partial-dual orientable genus polynomials for ribbon graphs, which count their partial duals by their genus. They conjectured that no orientable ribbon graph has a non-constant partial-dual polynomial with one non-zero term. This conjecture was refuted by an infinite family of counterexamples in \cite{MR4185110}. Chmutov and Vignes-Tourneret \cite{MR4271645} showed that these are the only counterexamples, and raised the question of whether the partial-dual polynomials and conjectures make sense for general delta-matroids. Yan and Jin \cite{MR4420997} introduced the twist polynomials for delta-matroids, which are analogous to the partial-dual polynomials for ribbon graphs. They also characterized the even normal binary delta-matroids with one term twist polynomials, and solved the odd normal binary case in \cite{Qi2022Xian}. They partially answered the question for normal binary delta-matroids, and left open the question for non-binary delta-matroids.

We organize the paper as follows. In Section \ref{sect:prelimi}, we recall the definitions and properties of delta-matroids, partial-duality polynomials of ribbon graphs and delta-matroids, and some other basic concepts. In Section \ref{mainresults}, we prove our main result for even normal non-binary delta-matroids: none of them has a non-constant twist polynomial with one non-zero term.

\section{Preliminaries}
\label{sect:prelimi}
Let $D$ be a $\Delta$-matroid with a finite ground set $E$ and a collection $\cF$ of subsets of $E$ called \emph{feasible sets}, which satisfy the following condition:

(SEA) For any $F_1, F_2 \in \cF$ and $x \in F_1\Delta F_2$, we have $F_1\Delta {x, y} \in \cF$ whenever $y \in F_2\Delta F_1$. Note that $x=y$ is allowed.

Given a delta-matroid $D=(E, \cF)$, the largest feasible sets of $D$ form the bases of the \emph{upper matroid}, while the smallest feasible sets of $D$ form the bases of the \emph{lower matroid}. These are two matroids that are contained in $D$.

Every $\Delta$-matroid $D = (E, \cF)$ has a \emph{dual} $\Delta$-matroid $D^\star = (E, \cF^\star)$, where $\cF^\star = {E\setminus F | F \in \cF}$. An element of $E$ that belongs to no feasible set of $D$ is a \emph{loop} of $D$, while an element of $E$ that belongs to no feasible set of $D^\star$ is a \emph{coloop} of $D$. Observe that the lower (upper) matroid of $D$ is dual to the upper (lower) matroid of $D^\star$.

\begin{definition}[Elementary minors] \label{def:minor}
Let $D=(E, \cF)$ be a delta-matroid. The elementary minors of $D$ at $e\in E$, are the delta-matroids $D-e$ and $D/e$ defined by:
$$D-e=\Big(E-e, \big\{F | F\subseteq E-e, F\in \cF\big\}\Big),$$
if $e$ is not a coloop, and 
$$D/e=\Big(E-e, \big\{F | F\subseteq E-e, F\cup e\in \cF\big\}\Big),$$
if $e$ is not a loop. In case $e$ is a loop or a coloop, we set $D/e=D-e$.
The delta-matroid $D-e$ is called the deletion of $D$ along $e$, and $D/e$ the contraction of $D$ along $e$.
\end{definition}

A \emph{minor} of $D$ is a $\Delta$-matroid that is obtained from a $\Delta$-matroid $D$ by a (potentially empty) sequence of contractions and deletions.

Assume that $A[W]=(a_{vw}: v, w\in W)$ for $W\subseteq E$ and that $A=(a_{vw}: v, w\in E)$ is a symmetric binary matrix. $D(A)=(E, \{W: A[W] \text{ has an inverse}\})$ is a $\Delta$-matroid if and only if $A[\emptyset]$ has an inverse.

\begin{definition}[Twist]\label{def:twist}
Let $D=(E, \cF)$ be a set system. For $A\subseteq E$, the twist of $D$ with respect to $A$, denoted by $D\star A$, is given by $(E, \{A\Delta X | X\in \cF\})$.
\end{definition}

We recall that $D^\star = D\star E$ is the dual $D^\star$ of $D$.
\begin{definition}[Binary delta-matroid \cite{MR1020647, MR1083464}]\label{def:binarydeltamat}
A  delta-matroid $D=D(E, \cF)$ is said to be binary if there exists $F\in \cF$ and a symmetric binary matrix $A$ such that $D=D(A)\star F$.
\end{definition}

A series of contractions and deletions results in the minor of a delta-matroid $D$. The following proposition where introduced in \cite{MR1020647,MR1083464}

\begin{prop}
If $D$ is a binary delta-matroid, then every elementary minor of $D$ is also a binary delta-matroid.
\end{prop}

\begin{prop}\label{prop:minimalnonbinary} A delta-matroid is binary if it has no minor
isomorphic to a twist of $S_1$, $S_2$, $S_3$, $S_4$, or $S_5$, where

\begin{eqnarray}
&&S_1=\Big(\{1,2,3\}, \big\{\emptyset, \{1,2\}, \{1,3\}, \{2,3\}, \{1,2,3\}\big\}\Big),\cr\cr
&&S_2= \Big(\{1,2,3\}, \big\{\emptyset, \{1\}, \{2\}, \{3\}, \{1,2\}, \{1,3\}, \{2,3\}\big\}\Big),\cr\cr
&&S_3= \Big(\{1,2,3\}, \big\{\emptyset, \{2\}, \{3\}, \{1,2\}, \{1,3\}, \{1,2,3\}\big\}\Big),\cr\cr
&&S_4= \Big(\{1,2,3,4\}, \big\{\emptyset, \{1, 2\}, \{1,3\}, \{1,4\}, \{2,3\}, \{2,4\}, \{3,4\}\big\}\Big),\cr\cr
&&S_5= \Big(\{1,2,3,4\}, \big\{\emptyset, \{1, 2\}, \{1,4\}, \{2,3\}, \{3,4\}, \{1,2,3,4\}\big\}\Big).
\end{eqnarray}
\end{prop}

We will only focus on even delta-matroids and will be restricted to $S_4$ and $S_5$ because the minor of an even delta-matroid is always even. 

\begin{prop}\label{prop:twistdelet}
For any delta-matroid $D=D(E, \cF)$, $e\in E$ and $F\subset E$, we have
\begin{enumerate}
    \item $(D\star F)/e=(D/e)\star F$ if $e\notin F$,
    \item $(D\star F)/e=(D\setminus e)\star (F\setminus e)$ if $e\in F$,
    \item $(D\star F)\setminus e=(D\setminus e)\star F$ if $e\notin F$,
    \item $(D\star F)\setminus e=(D/e)\star (F\setminus e)$ if $e\notin F$. 
\end{enumerate}

\end{prop}

\begin{definition}[Partial-dual orientable polynomial for delta-matroids \cite{MR4056111}]\label{def:partdualdeltamat}
Let $D=(E, \cF)$ be a delta-matroid. The partial-dual Euler-genus polynomial for $D$ is the generating function
\bea
\partial_{\Gamma_D}(z)=\sum_{A\subseteq E}z^{w(D\star A)}.
\eea
\end{definition}

\section{Main results}
\label{mainresults}
\begin{thm}\label{theo:evenset}
Let $E$ be a finite set and $\cF$ the set of subsets of $E$ with even cardinality. The pair
$D=(E, \cF)$ is a $\Delta$-matroid.
\end{thm}

\proof
Let $F_1$ and $F_2$ be two sets in $\cF$ that are symmetrically different, i.e., $F_1\Delta F_2\neq \emptyset$. We want to find another set $F_3$ in $\cF$ that is obtained by swapping two elements between $F_1$ and $F_2$. To do this, we pick any $x$ in $F_1\Delta F_2$ and look for a $y$ in $F_1\Delta F_2$ such that $y\neq x$. Then we define $F_3$ as $F_1\Delta \{x, y\}$. This means that we either remove or add $x$ and $y$ to $F_1$, depending on whether they belong to $F_1$ or not. We can show that $F_3$ is always in $\cF$ by considering three cases:

$\bullet$  If $x$ and $y$ are both in $F_1$, then $F_3=F_1\setminus \{x, y\}$, which is in $\cF$ because $F_3$ is even.

$\bullet$  If $x$ and $y$ are both in $F_2$, then $F_3=F_1\cup \{x, y\}$, which is in $\cF$ because  $F_3$ is even.

$\bullet$  If $x$ is in $F_1$ but not in $F_2$, and $y$ is in $F_2$ but not in $F_1$, then $F_3$ has the same cardinality as $F_1$, which is even by assumption. Therefore, $F_3$ is in $\cF$ because $\cF$ only contains sets of even cardinality.

Note that we cannot have $F_1\Delta F_2=\{x\}$ for some $x$, because that would imply that $F_1$ and $F_2$ differ by only one element, which is impossible since they have even cardinality. Hence, we can always find a $y$ in $F_1\Delta F_2$ that is different from $x$.
\qed

Let's call the delta-matroid in Theorem \ref{theo:evenset} $D^n$, where the empty set is a feasible set, the ground set has $n$ elements, and $n$ is an odd number.

\begin{thm}\label{theo:partdual}
Let $D^n=(E, \cF)$ be the delta-matroid defined above and $A\subseteq E$.
\begin{enumerate}
    \item If $A$ has even number of elements then $\rk((D^n\star A)_{min})=0$ and $\rk((D^n\star A)_{max})=n-1$.
    \item Otherwise, $\rk((D^n\star A)_{min})=1$ and $\rk((D^n\star A)_{max})=n$.
\end{enumerate}
\end{thm}
\proof
Let us start with the first point. If $A\subseteq E$ has even number of elements then $A\in \cF$ and $D^n\star A$ has the empty set as feasible and then $\rk((D^n\star A)_{min})=0$. The set $E\setminus A$ is odd and for $x\in E\setminus A$, $(E\setminus A)-x\in \cF$ and therefore $A\Delta (E\setminus A)-x\in \cF$ and contains $n-1$ elements. This ends the proof of 1).

We now turn to the case where $A$ is odd. In this case $A\notin \cF$ and then the smallest feasible set in $D^n\star A$ is non-empty and for any $x\in A$, $\{x\}\in \cF(D^n\star A)$ because $A-x\in \cF$. Therefore $\rk((D^n\star A)_{min})=1$. Since $A$ is odd then $E\setminus A\in \cF$ because it contains an even number of elements and therefore $A\Delta (E\setminus A)=E\in D^n\star A$. Hence $\rk((D^n\star A)_{max})=n$
\qed

Let us denote by $w(D^n)=\rk(D^n_{min})-\rk(D^n_{max})$, the width of $D^n$. The following results are immediate.

\begin{cor}
The evaluation of the partial-dual polynomial on the delta-matroids $D^n$ is  $\partial_{\Gamma_D}(z)=2^nz^{\frac{n-1}{2}}$.
\end{cor} 
This corollary demonstrates that the delta-matroids $D^n$ serve as natural expansions to the set of counterexamples presented in \cite{MR4185110}.

\begin{prop}
    Let $D=(E, \cF)$ be a delta-matroid and $A\subset E$. The delta-matroid $D\star A$ obtained by taking a twist of $D$ by $A$ is even (resp odd) if and only if $D$ is even (resp odd). In the same way, a minor of $D$ is even (resp odd) if $D$ is even (resp odd). 
\end{prop}

\begin{prop}\label{prop:used} Let $D=(E, \cF)$ be a delta-matroid in which the emptyset is a feasible satisfying 
$w(D)=w(D\star A)$ for any $A\subseteq E$.
\begin{enumerate}    
    \item If $D$ is even then there is no element $x$ in $E$ that belongs to every $F\in \cF_{max}$. Furthermore, if $\{a, b\}\in \cF$, then for any $F\in\cF_{max}$, $a\in F$ or $b\in F$.
    \item If $E\in \cF$, then $\cF=\mathcal{P}(E)$.
\end{enumerate}
\end{prop}
The first item of this proposition shows that the feasible set $\cF$ contains all the two elements subsets $\{a, b\}$ such that there is $F\in \cF_{max}$ and  $a\in F$ or $b\in F$. 
\proof
If there is an element $x\in E$ such that $x\in F$ for any $F\in \cF_{max}$, then $w(D\star \{x\})=w(D)-2$ because $\{x\}\notin \cF$ and there is no feasible of size $|F|-1$ in $\cF$. In case there is $\{a, b\}\in \cF$ and $F\in \cF{max}$ such that $a,b \notin F$ then $w(D\star \{a, b\})=w(D)+2$.

Assume that $E\in \cF$. If $A\subset E$ such that $A\notin \cF$ then $w(D\star A)=|E|-k<|E|=w(D)$ with $k=r((D\star A)_{min})>0$. Therefore $\cF=\mathcal{P}(E)$.
\qed

\begin{lemma}\label{lemma:atleastoneortwo}
Let's consider a matroid $M=(E, \cB)$ defined by its base set. For a base $F\in \cB$, elements $x, x'\in F$, and elements $y, y'\in E\setminus F$, we observe that if $F\Delta \{x, y\}$ and $F\Delta \{x', y'\}$ are in $\cB$, then we encounter two scenarios: either $F\Delta \{x, y\}\Delta \{x', y'\}$ is in $\cB$, or both $F\Delta \{x, y'\}$ and $F\Delta \{x', y\}$ are in $\cB$. Additionally, if $\cB$ includes a set of the form $F\Delta \{x,y\}\Delta\{x',y'\}$, then for any $\alpha, \beta \in \{y, y'\}$, the sets $F\Delta \{x,\alpha\}$ and $F\Delta \{x',\beta\}$ are also in $\cB$.
\end{lemma}

\proof
Examining $F\Delta\{x, y\}$ and $F\Delta\{x', y'\}$ as members of $\cB$, and noting that $x'\in F\Delta\{x, y\}\setminus F\Delta\{x', y'\}$, we deduce the existence of an element $a\in F\Delta\{x', y'\}\setminus F\Delta\{x, y\}$ such that the symmetric difference $(F\Delta\{x, y\})\Delta \{x', a\}$ is in $\cB$. Consequently, $a$ must be either $x$ or $y'$, leading to the conclusion that either $F\Delta\{x', y\}$ or $F\Delta\{x, y\}\Delta\{x', y'\}$ is in $\cB$. Similarly, since $x\in F\Delta\{x', y'\}\setminus F\Delta\{x, y\}$, there must be an element $b\in F\Delta\{x, y\}\setminus F\Delta\{x', y'\}$ such that the symmetric difference $F\Delta\{x', y'\}\Delta \{x, b\}$ is in $\cB$; here, $b$ can be either $x'$ or $y$. This implies that either $F\Delta\{x, y'\}$ or $F\Delta\{x, y\}\Delta\{x', y'\}$ is in $\cB$.

This argument completes the proof of the first part of the lemma. The proof of the second part follows a similar logic, considering the sets $F$ and $F\Delta\{x,y\}\Delta\{x',y'\}$ within $\cB$ and applying the Symmetric Exchange Axiom (SEA).
\qed

\begin{remark}$\bullet$ It is not hard to see that for $i=1, \cdots 5$, there is a subset $A$ of the ground set for which $w(S_i\star A)\neq w(A)$.

$\bullet$ Remark that the feasible sets of the minimal delta-matroids $S_4$ and $S_5\star\{1,3\}$ are respectively of the form: $\{\emptyset, F, F\Delta\{1, 3\}, F\Delta\{2, 4\}, F\Delta\{1,4\}, F\Delta\{2, 3\}$, $F\Delta\{1, 3\}\Delta \{2, 4\}\}$;  and $\{F, F\Delta\{1, 3\}, F\Delta\{2, 4\}$, $F\Delta\{1,4\}, F\Delta\{2, 3\}$, $F\Delta\{1, 3\}\Delta \{2, 4\}\}$  for $F=\{1,2\}$. Therefore if $S_4$ or $S_5$ is a minor of a delta-matroid $D=(E, \cF)$ then it contains feasibles of the form $F$, $F\Delta \{x_1, x'_1\}$, $F\Delta \{x_2, x'_2\}$, $F\Delta \{x_1, x'_2\}$, $F\Delta \{x_2, x'_1\}$, $F\Delta \{x_1, x'_1\}\Delta \{x_2, x'_2\}$   where $x_1, x_2\in F$ and $x'_1, x'_2\in E\setminus F$.
\end{remark}

\begin{prop}\label{prop:conatinsfeas}
\begin{itemize}
    Let $D=(E, \cF)$ be  delta-matroid and $e\in E$. 
    \item[i)] If the delta-matroids $D\setminus e$ or $D/e$ contain feasibles of the form $F$, $F\Delta \{x_1, x'_1\}$, $F\Delta \{x_2, x'_2\}$, $F\Delta \{x_1, x'_2\}$, $F\Delta \{x_2, x'_1\}$, $F\Delta \{x_1, x'_1\}\Delta \{x_2, x'_2\}$   with $x_1, x_2\in F$ and $x'_1, x'_2\in E\setminus F$, the delta-matroid $D$ contain feasibles of the same form. Furthermore if $F$ is of maximum size in $D\setminus e$ or $D/e$, its correspondence in $D$ belongs to $D_{max}$.
    
    \item[ii)] If the delta-matroids $D\setminus e$ or $D/e$ is isomorphic to a twist of $S_i$; $i=4,5$, then there is a subset $A$ of $E$ such that $D\star A$ contains feasibles of the form $F$, $F\Delta \{x_1, x'_1\}$, $F\Delta \{x_2, x'_2\}$, $F\Delta \{x_1, x'_2\}$, $F\Delta \{x_2, x'_1\}$, $F\Delta \{x_1, x'_1\}\Delta \{x_2, x'_2\}$   with $x_1, x_2\in F$ and $x'_1, x'_2\in E\setminus F$.
\end{itemize}
\end{prop}

\proof
We first consider $D\setminus e$ and suppose that it contains feasible sets of the form $F$, $F\Delta \{x_1, x'_1\}$, $F\Delta \{x_2, x'_2\}$, $F\Delta \{x_1, x'_2\}$, $F\Delta \{x_2, x'_1\}$, $F\Delta \{x_1, x'_1\}\Delta \{x_2, x'_2\}$   with $x_1, x_2\in F$ and $x'_1, x'_2\in E\setminus F$. Definition \ref{def:minor} implies that none of these sets contains $e$ and they all belong to $\cF$ and there are bases of $D_{max}$ if and only if $F$ is of maximum size in $D\setminus e$. 

Let us turn to the case of $D/e$ and assume that it contains feasibles of the form $F$, $F\Delta \{x_1, x'_1\}$, $F\Delta \{x_2, x'_2\}$, $F\Delta \{x_1, x'_2\}$, $F\Delta \{x_2, x'_1\}$, $F\Delta \{x_1, x'_1\}\Delta \{x_2, x'_2\}$   with $x_1, x_2\in F$ and $x'_1, x'_2\in E\setminus F$. Applying the definition of contraction it results that $D$ contains the feasibles: $F'$, $F'\Delta \{x_1, x'_1\}$, $F'\Delta \{x_2, x'_2\}$, $F'\Delta \{x_1, x'_2\}$, $F'\Delta \{x_2, x'_1\}$, $F'\Delta \{x_1, x'_1\}\Delta \{x_2, x'_2\}$; $F'=F\cup \{e\}$ $x_1, x_2\in F$ and $x'_1, x'_2\in E\setminus F$. There are clearly bases of $D_{max}$ if and only if $F$ is of maximum size in $D/ e$.  This ends the proof of the first item.

For the second item, if $D\setminus e$ is isomorphic to a twist of $S_i$; $i=4, 5$, then there is a subset $A$ of $E\setminus e$ such that  $D\setminus e$ is isomorphic to $S_i\star A$. Proposition \ref{prop:twistdelet} implies that $(D\setminus e)\star A=(D\star A)\setminus e$ is isomorphic to $S_i$. Using the result in the first item, $D\star A$ contains the feasible sets of the form: $F$, $F\Delta \{x_1, x'_1\}$, $F\Delta \{x_2, x'_2\}$, $F\Delta \{x_1, x'_2\}$, $F\Delta \{x_2, x'_1\}$, $F\Delta \{x_1, x'_1\}\Delta \{x_2, x'_2\}$; $x_1, x_2\in F$ and $x'_1, x'_2\in E\setminus F$. We obtain same result by replacing the deletion by a contraction.
\qed

 The results in Proposition \ref{prop:conatinsfeas} also applies for any minor $M$ associated to a given  delta-matroid $D$. Meaning that if $M$ contains feasible sets of the form $F$, $F\Delta \{x_1, x'_1\}$, $F\Delta \{x_2, x'_2\}$, $F\Delta \{x_1, x'_2\}$, $F\Delta \{x_2, x'_1\}$, $F\Delta \{x_1, x'_1\}\Delta \{x_2, x'_2\}$; $x_1, x_2\in F$ and $x'_1, x'_2\in E\setminus F$, then $D$ also has feasible sets of the same form. Furthermore they belong to $D_{max}$ if there correspondence in $M$ belong to $M_{max}$ where $M_{max}$ is the upper matroid associated to $M$. 

\begin{thm}
There is no even non binary $\Delta$-matroid  $D=(E, \cF)$ in which the emptyset is a feasible satisfying 
$w(D)=w(D\star A)$ for any $A\subseteq E$.
\end{thm}

\begin{proof}
We assume for simplicity that every $x\in E$ is in some feasible set of $D$. Otherwise, for any $x\in E$ such that $x\notin F$ for all $F\in \cF$, we have $w(D\star A)=w(D\star (A\setminus \{x\}))$ for any subset $A$ of $E$ that includes $x$.

$\bullet$ If $|E|=1$, then $\cF=\mathcal{P}(E)$, which cannot be, since $(E, \mathcal{P}(E))$ is not an even $\Delta$-matroid.

$\bullet$ Let $E=\{x_1, x_2\}$. If $r(D_{max})=1$, we use the previous result. Otherwise, $\cF_{max}$ has all the subsets of $E$ with one element and $w(D)=1\neq 2=w(D\star \{x_1\})$. $\mathcal{F}_{max}$ cannot have a subset of size 2, because then the second point in Proposition \ref{prop:used} would imply that $\cF=\mathcal{P}(E)$, which is not an even $\Delta$-matroid. 

$\bullet$ Let $E=\{x_1, x_2, x_3\}$. As in the previous case, $\mathcal{F}{max}$ cannot have sets of size 3, or else $\cF=\mathcal{P}(E)$. Now suppose that $r(D_{max})=2$. The only non-binary $\Delta$-matroid is $S_2$, but $w(S_2\Delta \{1\})=3$, which is impossible.

$\bullet$ Let $E=\{x_1, x_2, x'_1, x'_2\}$. Assume that $r(D{max})=2$. Since $D$ is an even non-binary, it has feasible sets of the form $S_4$ or $S_5$, but $w(S_4)=2\neq 4=w(S_4\star {1, 2})$ and $w(S_5)=4\neq 0=w(S_5\star {1, 3})$. In fact, if $r(D{max})=2$, then $D$ has feasible sets of the form $F=\{x_1, x_2\}$, $F\Delta\{x_1, x'_1\}$, $F\Delta\{x_2, x'_2\}$, $F\Delta\{x_1, x'_2\}$, $F\Delta\{x_2, x'_1\}$ and $F\Delta\{x_1, x'_1\}\Delta \{x_2, x'1\}$, which is isomorphic to $S_4$. If $r(D{max})=4$, then $D$ must be $\mathcal{P}(E)$, which is not an even $\Delta$-matroid.

$\bullet$ Now let $E=\{x_1, x_2, x_3, x_4, x'_1\}$. If $r(D{max})=2$, the previous case applies and works for any delta-matroid with more than four elements in the ground set. We assume that $r(D_{max})=4$. From Proposition \ref{prop:used}, $\cF=\{\emptyset, \{x_1, x_2\}, \{x_1, x_3\}, \{x_1, x_4\}, \{x_1, x_4\}$, $\{x_1, x'_1\},  \{x_2, x_3\}, \{x_2, x_4\}, \{x_2, x'_1\}, \{x_3, x_4\}, \{x_3, x'_1\}, \{x_4, x'_1\}, F, F\Delta\{x_1, x'_1\}, F\Delta\{x_2, x'_1\}$, $F\Delta\{x_3, x'_1\}, F\Delta\{x_4, x'_1\}\}$ with $F=\{x_1, x_2, x_3, x_4\}$, but this delta-matroid is binary.

$\bullet$ Suppose that $E=\{x_1, x_2, x_3, x_4, x'_1, x'_2\}$ and $r(D{max})=4$ with $F=\{x_1, x_2, x_3, x_4\}\in \cF_{max}$. Since $D$ is non-binary, it has feasible sets of the form $F$, $F\Delta\{x_1, x'_1\}$, $F\Delta\{x_2, x'_2\}$, $F\Delta\{x_1, x'_2\}$, $F\Delta\{x_2, x'_1\}$ and $F\Delta\{x_1, x'_2\}\Delta \{x_2, x'_2\}$. According to Proposition \ref{prop:used}, $x_j$; $j=3, 4$ does not belong to a feasible set in $\cF_{max}$. This leaves us with two options based on the SEA: 1) $F\Delta\{x_3, x'_1\}$, $F\Delta\{x_4, x'_1\}$ are feasible sets or 2) $F\Delta\{x_3, x'_1\}$, $F\Delta\{x_4, x'_2\}$ $\in \cF$. The same proposition implies that $\{x'_1, x'_2\}$, $\{x_1, x'_2\}$, $\{x_2, x'_1\}$, $\{x_1, x'_1\}$, $\{x_2, x'_2\}$ and $\{x_1, x_2\}$ $\notin \cF$.

The first case leads to the fact that $\{x'_2, x_1\}$, $\{x'_2, x_3\}$ and  $\{x'_2, x_4\}$ $\notin \cF$, which shows that $F\Delta\{x_2, x'_2\}=\{x_1, x'_2, x_3, x_4\}$ cannot be a feasible set, because it violates the SEA on the empty set and $\{x_1, x'_2, x_3, x_4\}$. Since $D_{max}$ is a matroid, applying Lemma \ref{lemma:atleastoneortwo} on the second case gives: $F\Delta\{x_3, x'_2\}$, $F\Delta\{x_4, x'_1\}$ $\in \cF$ (which goes back to the first case) or $F\Delta{x_3, x'_1}\Delta{x_4, x'_2}\in \cF$.
We consider the case where only $F\Delta\{x_3, x'_1\}\Delta\{x_4, x'_2\}\in \cF$, which implies that $\{x_3, x_4\}\notin \cF$. Applying Lemma \ref{lemma:atleastoneortwo} again on $F\Delta\{x_3, x'_1\}$ and $F\Delta\{x_1, x'_2\}$ or $F\Delta\{x_2, x'_2\}$, we return to the first case or we get $F\Delta\{x_3, x'_1\}\Delta\{x_1, x'_2\}\in \cF$ and $F\Delta\{x_3, x'_1\}\Delta\{x_2, x'_2\}\in \cF$, which implies that $\{x_3,x_1\}, \{x_3,x_2\}\notin \cF$. This contradicts the fact that the SEA between $F$ and the empty set should give at least one of the following: $\{x_3,x_1\}, \{x_3,x_2\}, \{x_3,x_4\}\in \cF$.

$\bullet$ We now assume that $E=\{x_1, x_2, x_3, x_4, x'_1, x'_2, x'_3\}$ and $r(D_{max})=4$ with $F=\{x_1, x_2, x_3, x_4\}\in \cF_{max}$. Since $D$ is non binary, its has feasible set of the form:  $F$, $F\Delta\{x_1, x'_1\}$, $F\Delta\{x_2, x'_2\}$, $F\Delta\{x_1, x'_2\}$, $F\Delta\{x_2, x'_1\}$ and $F\Delta\{x_1, x'_2\}\Delta \{x_2, x'_2\}$. Applying Proposition \ref{prop:used} and the SEA, we have the following possibilities: 1) $F\Delta\{x_3, x'_1\}$, $F\Delta\{x_4, x'_1\}$ $\in \cF$, 2) $F\Delta\{x_3, x'_1\}$, $F\Delta\{x_4, x'_2\}$ are feasible sets, 3) $F\Delta\{x_3, x'_3\}$, $F\Delta\{x_4, x'_3\}$ $\in \cF$ or 4) $F\Delta\{x_3, x'_1\}$, $F\Delta\{x_4, x'_3\}$ $\in\cF$. The same proposition implies that no pair in $\{x'_1, x'_2, x'_3\}$, $\{x_1, x'_2, x'_3\}$, $\{x_2, x'_1, x'_3\}$, $\{x_1, x'_1, x'_3\}$, $\{x_2, x'_2, x'_3\}$ and $\{x_1, x_2, x'_3\}$ belongs to $\cF$. Proceeding in a similar way as earlier, each of these cases breaks the SEA.    

    If $r(D_{max})=6$, the result follows the same analysis made in the case $|E|=5$,  $r(D_{max})=4$.

     $\bullet$ Consider $E=\{x_1, x_2, x_3, x_4, x_5, x_6, x'_1, x'_2\}$ and $r(D_{max})=4$ with $F=\{x_1, x_2, x_3, x_4\}$ element of $\cF_{max}$. Since $D$ is non binary, it has feasible sets of the form:  $F$, $F\Delta\{x_1, x'_1\}$, $F\Delta\{x_2, x'_2\}$, $F\Delta\{x_1, x'_2\}$, $F\Delta\{x_2, x'_1\}$ and $F\Delta\{x_1, x'_2\}\Delta \{x_2, x'_2\}$. Since $F\Delta\{x_1, x'_1\}$ and $ F\Delta\{x_2, x'_2\}$ belong to $\cF$, then $\{x_1, x_5\}, \{x_2, x_5\}\notin \cF$. Otherwise $W(D\star \{x_1, x_5\})=W(D)+2=W(D\star \{x_2, x_5\})$ because of the following relations $(F\Delta\{x_1, x'_1\})\Delta\{x_1, x_5\}=F\cup\{x_5, x'_1\}$ and $(F\Delta\{x_2, x'_2\})\Delta\{x_2, x_5\}=F\cup\{x_5, x'_2\}$.
     Using Proposition \ref{prop:used} and the SEA, $F\Delta\{x_3,\alpha\}$, $F\Delta\{x_4,\alpha\}\in \cF$ for $\alpha=x_5, x_6, x'_1, x'_2$. The case $F\Delta\{x_3,\alpha\}, F\Delta\{x_4,\alpha\}\in \cF$ for $\alpha= x'_1, x'_2$ is already studied in the previous case. Now assume that $F\Delta\{x_3, x_5\}\in \cF$ and $F\Delta\{x_4, x_6\}\in \cF$.
     This is impossible because $F\Delta\{x_4, x_6\}\in \cF$ implies that $\{x_5, x_4\}\notin \cF$ and therefore contradict the fact that $F\Delta\{x_3, x_5\}\in \cF$ since the SEA between the empty set and $F\Delta\{x_3, x_5\}$ implies that there should exist $\alpha=x_1, x_2, x_4$ such that $\{x_5, \alpha\}\in \cF$. If instead we have $F\Delta\{x_3, x_5\}, F\Delta\{x_4, x_5\}\in \cF$ then it contradicts the fact that $F\Delta\{x_1, x'_1\}\Delta\{x_2, x'_2\}\in \cF$ because $\{x'_2, x'_1\}, \{x'_2, x_3\}, \{x'_2, x_4\}\notin \cF$ from the fact that $F\in \cF$ and $F\Delta\{x_3, x_5\}, F\Delta\{x_4, x_5\}\in \cF$. If otherwise we have $F\Delta\{x_3, x_5\}, F\Delta\{x_4, x'_1\}\in \cF$ and $F\Delta\{x_3, x_5\}\Delta\{x_4, x'_1\}\notin \cF$ then $F\Delta\{x_3, x'_1\}, F\Delta\{x_4, x_5\}\in \cF$ which has just been studied earlier. Otherwise if $F\Delta\{x_3, x_5\}\Delta\{x_4, x'_1\}\in \cF$ then $\{x_4, x'_1\}, \{x_3, x'_1\}\notin \cF$. Furthermore $\{x'_2, x'_1\}\notin \cF$ contradicts the fact that $F\Delta\{x_1, x'_1\}\Delta\{x_2, x'_2\}\in \cF$. This is obtained by applying the SEA between the empty set and  $F\Delta\{x_1, x'_1\}\Delta\{x_2, x'_2\}=\{x'_1, x'_2, x_4, x_4\}$.

Let's assume that $r(D_{max})=6$. Given that $D$ is non-binary, we have the feasible sets $F$, $F\Delta \{x_1, x'_1\}$, $F\Delta \{x_2, x'_2\}$, $F\Delta \{x_1, x'_2\}$, $F\Delta \{x_2, x'_1\}$, and $F\Delta \{x_1, x'_2\}\Delta \{x_2, x'_2\}$ within $\cF$, where $F$ is given by $F=\{x_1, x_2, x_3, x_4, x_5, x_6\}$. It follows that none of the sets $\{x'_1, x'_2\}$, $\{x_1, x'_2\}$, $\{x_2, x'_1\}$, $\{x_1, x_2\}$, $\{x_1, x'_1\}$, or $\{x_2, x'_2\}$ are in $\cF$. According to Proposition \ref{prop:used} and the SEA, for each $i=3,4,5,6$ and $\alpha=x'_1,x'2$, the set $F\Delta \{x_i,\alpha\}$ is in $\cF$. We can consider two cases without loss of generality:

$1)$ The sets $F\Delta \{x_3, x'_1\}$, $F\Delta \{x_4, x'_1\}$, $F\Delta \{x_5, x'_1\}$, and $F\Delta \{x_6, x'_2\}$ are feasible.

$2)$  The sets $F\Delta \{x_3, x'_1\}$, $F\Delta \{x_4, x'_1\}$, $F\Delta\{x_5, x'_1\}$, and $F\Delta \{x_6, x'_2\}$ are in $\cF$.

In the first case, this means that the sets $\{x_3, x'_2\}$, $\{x_4, x'_2\}$, $\{x_5, x'_2\}$, and $\{x_6, x'_1\}$ are not in $\cF$, which implies that $\{x_6, x'_2\}$ must be in $\cF$ because the symmetric difference $F\Delta \{x_2, x'_2\}$ results in $\{x_1, x'_2, x_3, x_4, x_5, x_6\}$ and and the empty set belong to $\cF$. However, having both $F\Delta \{x_6, x'_2\}$ and $F\Delta \{x_1, x'_1\}$ in $\cF$ would necessitate that either $F\Delta \{x_6, x'_1\}$ is in $\cF$ (which cannot be since $\{x_6, x'_2\}$ is in $\cF$) or the symmetric difference $F\Delta \{x_6, x'_2\}\Delta \{x_1, x'_1\}$, which equals $\{x'_1, x_2, x_3, x_4, x_5, x'_2\}$, is in $\cF$. The latter is also not possible because none of the sets $\{x'_2, \alpha\}$ for $\alpha \in \{x'_1, x_2, x_3, x_4, x_5\}$ are in $\cF$.

Now, let's explore the more general case where $E= \{x_1, \dots, x_m, x'_1, \dots, x'_p\}$, meaning $|E|=m+p$, and $F=\{x_1, \dots, x_m\}$. Since $D$ is non-binary, it includes feasible sets of the form $F$, $F\Delta \{x_1, x'_1\}$, $F\Delta \{x_2, x'_2\}$, $F\Delta \{x_1, x'_2\}$, $F\Delta \{x_2, x'_1\}$, and $F\Delta \{x_1, x'_2\}\Delta \{x_2, x'_2\}$.

From the assumption we made from the beginning of the proof, there is a feasible set $F_j\in \cF$ such that $x'_j\in F_j$ for any $j=3, \cdots, m$. The SEA implies that for each  $i=1, \cdots, m$ there is $j=1, \cdots, p$  such that $F\Delta \{x_i, x'_j\}\in \cF$.

Let's consider that for each $i$ from 1 to $m$, there exists a $j$ from 3 to $p$ such that the symmetric difference $F\Delta \{x_i, x'_j\}$ is included in $\cF$. This implies that the set $\{x_i, x'_1\}$ is not in $\cF$ for any $i$ in the range from 1 to $m$. If it were otherwise, applying the Symmetric Exchange Axiom (SEA) to $\{x_i, x'_1\}$ and $F\Delta \{x_i, x'_j\}$ would lead to the inclusion of $F\cup \{x'_1, x'_j\}$ in $\cF$, which cannot occur. Moreover, the assertion that $\{x_i, x'_1\}$ is excluded from $\cF$ for each $i$ from 1 to $m$ contradicts the fact that $F\Delta \{x_1, x'_1\}$ is a member of $\cF$. Furthermore, if for all $i$ from 1 to $m$, the set $F\Delta \{x_i, x'_1\}$ belongs to $\cF$, then the set $\{x_i, x'_2\}$ must not be in $\cF$, which would be in conflict with the established fact that $F\Delta \{x_2, x'_2\}$ is part of $\cF$.

Suppose, without loss of generality, that for some $q$ within the set $\{1, \ldots, m\}$, the set $F\Delta \{x_q, x'_2\}$ is in $\cF$. However, for all $i$ from 1 to $m$, excluding $q$, the set $F\Delta \{x_i, x'_1\}$ is also in $\cF$. Considering both $F\Delta \{x_q, x'_2\}$ and $F\Delta \{x_i, x'_1\}$ for all $i$ not equal to $q$, and applying Lemma \ref{lemma:atleastoneortwo}, we find that either $F\Delta \{x_q, x'_1\}$ is in $\cF$ or the symmetric difference $F\Delta {x_q, x'_2}\Delta {x_i, x'_1}$ is in $\cF$ for any $i$ in the set $\{1, \ldots, m\}$ excluding $q$. The former case circles back to a previously examined scenario. In the latter case, the pair $\{x_q, x_i\}$ cannot be in $\cF$ for any $i$ not equal to $q$, which contradicts the SEA applied to the empty set and $F$.

Next, let's assume there exist $q, r$ within the set $\{1, \ldots, m\}$, excluding $\{1,2\}$, and distinct from each other, such that both $F\Delta \{x_q, x'_2\}$ and $F\Delta \{x_r, x'_2\}$ are in $\cF$, but for all $i$ from 1 to $m$, excluding $q$ and $r$, the set $F\Delta \{x_i, x'_1\}$ is in $\cF$. Consequently, the sets $F\Delta \{x_1, x'_2\}$, $F\Delta \{x_2, x'_2\}$, $F\Delta \{x_q, x'_2\}$, and $F\Delta \{x_r, x'_2\}$ are the sole members of $\cF$ that can be expressed as $F\Delta \{\alpha, x'_2\}$. This means that for each $F\Delta \{x_i, x'_1\}$, where $i$ is in the set $\{2, \ldots, m\}$ excluding $q$ and $r$, and for $F\Delta \{x_1, x'_2\}$, the symmetric difference $F\Delta \{x_i, x'_1\}\Delta \{x_1, x'_2\}$ is in $\cF$ according to Lemma \ref{lemma:atleastoneortwo}. Thus, the pair $\{x_1, x_i\}$ is not in $\cF$ for any $i$ in the set $\{2, \ldots, m\}$ excluding $q$ and $r$. If either $F\Delta \{x_q, x'_1\}$ or $F\Delta \{x_r, x'_1\}$ is in $\cF$, we revert to the previous case. If not, considering $F\Delta \{x_1, x'_1\}$ and each $F\Delta \{x_j, x'_2\}$ for $j=q, r$, Lemma \ref{lemma:atleastoneortwo} implies that the symmetric difference $F\Delta \{x_1, x'_1\}\Delta \{x_j, x'_2\}$ is in $\cF$ for $j= q, r$, leading to the conclusion that the pair ${x_1, x_j}$ is not in $\cF$ for $j= q, r$. Ultimately, no pair of the form $\{x_1, x_i\}$, where $i$ ranges from 2 to $m$, exists in $\cF$, contradicting the SEA applied to $F$ and the empty set.

If we proceed inductively, we arrive at the same conclusion if $\cF$ includes additional elements of the form $F\Delta \{\alpha, x'_2\}$.
\end{proof}

\bibliographystyle{alpha}
\bibliography{references}

\begin{thebibliography}{GMT20}

\bibitem[BD91]{MR1083464}
A.~Bouchet and A.~Duchamp.
\newblock Representability of {$\triangle$}-matroids over {${\rm GF}(2)$}.
\newblock {\em Linear Algebra Appl.}, 146:67--78, 1991.

\bibitem[Bou89]{MR1020647}
Andr\'{e} Bouchet.
\newblock Maps and {$\triangle$}-matroids.
\newblock {\em Discrete Math.}, 78(1-2):59--71, 1989.

\bibitem[Chm09]{MR2507944}
Sergei Chmutov.
\newblock Generalized duality for graphs on surfaces and the signed
  {B}ollob\'{a}s-{R}iordan polynomial.
\newblock {\em J. Combin. Theory Ser. B}, 99(3):617--638, 2009.

\bibitem[CVT21]{MR4271645}
Sergei Chmutov and Fabien Vignes-Tourneret.
\newblock On a conjecture of {G}ross, {M}ansour and {T}ucker.
\newblock {\em European J. Combin.}, 97:Paper No. 103368, 7, 2021.

\bibitem[GMT20]{MR4056111}
Jonathan~L. Gross, Toufik Mansour, and Thomas~W. Tucker.
\newblock Partial duality for ribbon graphs, {I}: distributions.
\newblock {\em European J. Combin.}, 86:103084, 20, 2020.

\bibitem[QY22]{Qi2022Xian}
Xian’an~Jin Qi~Yan.
\newblock Twist monomials of binary delta-matroids, 2022.

\bibitem[YJ21]{MR4185110}
Qi~Yan and Xian'an Jin.
\newblock Counterexamples to a conjecture by {G}ross, {M}ansour and {T}ucker on
  partial-dual genus polynomials of ribbon graphs.
\newblock {\em European J. Combin.}, 93:Paper No. 103285, 12, 2021.

\bibitem[YJ22]{MR4420997}
Qi~Yan and Xian'an Jin.
\newblock Twist polynomials of delta-matroids.
\newblock {\em Adv. in Appl. Math.}, 139:Paper No. 102363, 12, 2022.

\end{thebibliography}
\end{document}